\documentclass{amsart}

\usepackage{times,amssymb,amsmath,mathrsfs,cite}

\def\tab(#1){\mbox{\small$\young(#1)$}\,}

\swapnumbers
\numberwithin{equation}{section}
\newtheorem{Defn}[equation]{Definition}
\newtheorem*{THEOREM}{Main Theorem}
\newtheorem{Theorem}[equation]{Theorem}
\newtheorem{Prop}[equation]{Proposition}

\newtheorem{Lemma}[equation]{Lemma}
\newtheorem{Cor}[equation]{Corollary}
\theoremstyle{definition}
\newtheorem{Remark}[equation]{Remark}

\def\map#1#2{\,{:}\,#1\!\longrightarrow\!#2}

\def\pmod#1{\text{ }(\text{mod } #1)\,}

{\catcode`\|=\active
  \gdef\set#1{\mathinner{\lbrace\,{\mathcode`\|"8000%
                                   \let|\midvert #1}\,\rbrace}}
}
\def\midvert{\egroup\mid\bgroup}
\let\<\langle
\let\>\rangle

\def\Sum{\displaystyle\sum}

\def\({\Big(}
\def\){\Big)}

\let\Gdom=\blacktriangleright
\DeclareMathOperator\Gedom{\,{\underline{\kern-.1ex{\blacktriangleright}\kern-0.1ex}}\,}
\let\gdom=\triangleright
\let\ldom=\triangleleft
\let\gedom=\trianglerighteq

\let\Sect=\S

\def\Bmu{\mathcal{B}_{\bmu,i}}
\def\N{\mathbb N}
\def\Z{\mathbb Z}

\def\charge{{\boldsymbol{\kappa}}}
\def\K{\mathcal K}
\def\O{\mathcal O}
\def\Sym{\mathfrak S}
\def\HH{\mathscr{H}}
\renewcommand\H[1][n]{\HH_{#1}}
\newcommand\HSym[1][n]{\HH_\xi (\Sym_n)}
\newcommand\HL[1][n]{\HH^\Lambda_{#1}}
\newcommand\HK[1][n]{\HH^\K_{#1}}
\newcommand\Hk[1][n]{\HH^K_{#1}}
\newcommand\HO[1][n]{\HH^\O_{#1}}
\newcommand\R[1][n]{\mathscr{R}^\Lambda_{#1}}
\def\Hmu{\HH_\bmu}

\def\bQ{\mathbf{Q}}
\def\bi{\mathbf{i}}
\def\bj{\mathbf{j}}

\def\imu{\bi^\bmu}
\def\ilmu{\bi_\bmu}

\def\S{\mathsf S}

\def\a{\mathfrak a}
\def\b{\mathfrak b}
\def\c{\mathfrak c}
\def\d{\mathfrak d}
\def\s{\mathfrak s}
\def\t{\mathfrak t}
\def\u{\mathfrak u}
\def\v{\mathfrak v}

\def\tlam{\t^\blam}

\def\bump{{\bmu'}}
\def\tnu{\t^\bnu}
\def\teta{\t^\bEta}
\def\tmu{\t^\bmu}
\def\tlmu{\t_\bmu}
\def\tmup{\t^\bump}
\def\tlmup{\t_\bump}
\def\zbmui{z_{\bmu\uparrow i}}
\def\zbmu{z_\bmu}

\def\balpha{{\boldsymbol{\alpha}}}
\def\blam{{\boldsymbol\lambda}}
\def\brho{{\boldsymbol\rho}}

\def\bEta{{\boldsymbol\eta}}
\def\bmu{{\boldsymbol\mu}}
\def\bnu{{\boldsymbol\nu}}
\def\bomega{{\boldsymbol{\omega}}}

\newcommand\talpmu[1][k]{\t^{\balpha_{#1}}_\bmu}

\def\RStd(#1){\mathop{\rm RStd}(#1)}
\def\Std(#1){\mathop{\rm Std}(#1)}
\def\SStd(#1){\mathop{\rm Std}\nolimits^2(#1)}

\newcommand{\Parts}[1][n]{\mathscr{P}_{#1}}

\DeclareMathAlphabet{\mathpzc}{OT1}{pzc}{m}{it}

\def\Mod{\textbf{Mod-}}

\DeclareMathOperator{\Ind}{Ind}
\DeclareMathOperator{\Shape}{Shape}
\DeclareMathOperator{\comp}{comp}
\DeclareMathOperator{\cont}{cont}
\DeclareMathOperator{\defect}{def}
\DeclareMathOperator{\iInd}{\text{$i$}-Ind}
\DeclareMathOperator{\Hom}{Hom}
\DeclareMathOperator{\res}{res}
\DeclareMathOperator{\codeg}{codeg}
\begin{document}
\bibliographystyle{andrew}
\title{Graded induction for Specht modules}
\subjclass[2000]{20C08, 20C30, 05E10}
\keywords{Cyclotomic Hecke algebras, Khovanov--Lauda--Rouquier algebras, cellular algebras}

\author{Jun Hu}
\address{School of Mathematics and Statistics F07,
University of Sydney, NSW 2006, Australia}
\email{junhu303@yahoo.com.cn}
\author{Andrew Mathas}
\address{School of Mathematics and Statistics F07,
University of Sydney, NSW 2006, Australia}
\email{a.mathas@usyd.edu.au}

\thanks{This research was supported, in part, by the Australian
Research Council}

\begin{abstract}
  Recently Brundan, Kleshchev and Wang introduced a $\Z$-grading on the Specht
modules of the degenerate and non-degenerate cyclotomic Hecke algebras of type
$G(\ell,1,n)$. In this paper we show that induced Specht modules have an explicit
filtration by shifts of graded Specht modules. This
proves a conjecture of Brundan, Kleshchev and Wang.
\end{abstract}

\maketitle

\section{Introduction}
In a remarkable series of papers Brundan and Kleshchev (and
Wang)~\cite{BK:GradedKL,BKW:GradedSpecht,BK:GradedDecomp} have shown that the
cyclotomic Hecke algebras of type $G(\ell,1,n)$ are $\Z$-graded algebras and they
have proved a graded analogue of Ariki's categorification theorem~\cite{Ariki:can}.
This work builds upon the work of Khovanov and Lauda~\cite[\Sect3.4]{KhovLaud:diagI} and
Rouquier~\cite{Rouq:2KM}'s introduction of the quiver Hecke algebras which are
certain graded algebras which categorify the negative part of quantum group of an
arbitrary Kac-Moody Lie algebra.

The representation theory of the cyclotomic Hecke algebras of type $G(\ell,1,n)$ is very well
developed with the Specht modules introduced in \cite{DJM:cyc,AMR} playing a central role. For each
multipartition $\bmu$ Brundan, Kleshchev and Wang~\cite{BKW:GradedSpecht} introduced a $\Z$-grading on
each Specht module $S^\bmu$ which is isomorphic to the (ungraded) Specht module upon forgetting the
grading.  In~\cite[Theorem~4.11]{BKW:GradedSpecht} they showed that restriction of the graded Specht
module has a graded Specht filtration. By Frobenius reciprocity in the Grothendieck group there is an
analogous formula for the induced graded Specht modules. Brundan, Kleshchev and Wang conjectured
that this should correspond to a filtration of the induced Specht module by shifts of graded Specht
modules. In this paper we prove this conjecture.

To state our main results, fix an integral domain $R$ and an integer $e\in\{0,2,3,4,\dots\}$ such that
either $e=0$ or $e$ is invertible in $R$ whenever $e$ is not prime, let $\HL$ be the graded cyclotomic
quiver Hecke algebra (over $R$) determined by $e$ and the dominant weight~$\Lambda$ (see
Definition~\ref{quiver Hecke relations}). The algebras $\HL$ include the cyclotomic Hecke algebras
of type $G(\ell,1,n)$ as special cases.  When $R$ is a field there is a natural graded embedding
$\HL\hookrightarrow\HL[n+1]$ which makes $\HL[n+1]$ into a free $\HL$-module by the main theorem of
\cite{HM:GradedCellular}. There is an induction functor
$$\Ind\map{\Mod\HL}\Mod\HL[n+1]; M\mapsto M\otimes_{\HL}\HL[n+1].$$
By projecting onto the blocks of $\HL[n+1]$ the induction functor decomposes as
$$\Ind = \bigoplus_{i\in I} \iInd,$$
where $I=\Z/e\Z$.

Fix $i\in I$ and let $\bmu$ be a multipartition of~$n$. Let
$\balpha_1,\dots,\balpha_z$ be the multipartitions of $n+1$ obtained
by adding an addable $i$-node to $\bmu$ (see before
Definition~\ref{gamma}), ordered so that
$\balpha_1\gdom\dots\gdom\balpha_z$, where $\gedom$ is the dominance
order on multipartitions (see \Sect2.1). Given an addable node~$A$
of $\bmu$ then we define the integer $d_A(\bmu)$ in
Definition~\ref{gamma} below. Finally, if $M$ is a graded
$\HL$-module and $d\in\Z$ then $M\<d\>$ is the graded $\HL$-module
obtained by shifting the grading on~$M$ by~$d$; see Section~4.

\begin{THEOREM}
  Suppose that $R$ is an integral domain and that either $e=0$ or $e$ is invertible in $R$
  whenever $e$ is not prime. Let $\bmu\in\Parts$ and $i\in I$. Then the induced graded Specht module
  $\iInd S^\bmu$ has a graded Specht filtration. That is, there exists a filtration
  $$0=I_0\subset I_1\subset\dots\subset I_z=\iInd S^\bmu,$$
  such that $I_j/I_{j-1}\cong S^{\balpha_j}\<d_{A_j}(\bmu)\>$.
\end{THEOREM}

Our arguments easily extend to the analogous result for the induced graded dual
Specht modules~$S_\bmu$; see Corollary~\ref{dual Spechts}.

When $e=0$ and $\ell=2$ this result can be deduced from \cite[Lemma~3.4]{BrundanStroppel:KhovanovIII}.
In the ungraded setting, when $\HL$ can be identified with the cyclotomic Hecke algebras of type
$G(\ell,1,n)$, this result was established by the second author~\cite{M:SpechtInduction}. The main
observation in~\cite{M:SpechtInduction} is that the filtration of the induced Specht module is the
restriction of a Specht filtration of a closely related family of `induced' modules $M(\bmu)$. In
general, we do not know how to construct graded lifts of the modules $M(\bmu)$ so we cannot use this
approach here.

To prove our Main Theorem we instead use a beautiful construction of
Ryom-Hansen which gives an explicit filtration of the induced Specht
modules for the Hecke algebras of the symmetric
group~\cite{RyomHansen:GradedTranslation} --- in the ungraded
setting Ryom-Hansen gave the first proof of our Main Theorem for the Hecke
algebras of the symmetric groups. In order to adapt Ryom-Hansen's
construction to the graded setting we use our recent construction of a
homogeneous cellular basis for $\HL$ and our realization, up to shift, of
the graded Specht modules as graded submodules of $\HL$; see~\cite{HM:GradedCellular}.

The outline of this paper is as follows. In the Chapter~2 we recall
the results that we need from the ungraded representation theory of
the Hecke algebras of type $G(\ell,1,n)$ and use this to establish a
strong result (Theorem~\ref{strong}), about the transition matrices
between the standard and seminormal bases of these algebras in the
semisimple case. In Chapter~3 we prove an analogous result
(Theorem~\ref{psi mn}), for the transition matrices between the
standard and homogeneous bases of the graded algebras. One
consequence of these results is a necessary condition for the
products $m_{\s\t}n_{\u\v}$ and $\psi_{\s\t}\psi'_{\u\v}$ to be
non-zero, where $\{m_{\s\t}\}$ and $\{n_{\u\v}\}$ are the standard
bases corresponding to the trivial and sign representations of $\H$
and $\{\psi_{\s\t}\}$ and $\{\psi'_{\u\v}\}$ are homogeneous analogues
of these bases. In Chapter~4 we use this non-vanishing condition,
together with the ideas of
Ryom-Hansen~\cite{RyomHansen:GradedTranslation}, to prove our Main
Theorem.

\section{Cyclotomic Hecke algebras and strong dominance} The aim of this paper is to understand the
effect of the graded induction functors on the Specht modules. To do this we first need to prove a
strong result about certain structure constants in the ungraded case.

\subsection{Cyclotomic Hecke algebras}
We start with the definition of the cyclotomic Hecke algebras of type
$G(\ell,1,n)$.

Fix an integral domain $R$ and an integer $\ell\ge1$. Define
$\delta_{\xi1}=1$ if $\xi=1$ and $\delta_{\xi1}=0$ otherwise.

\begin{Defn}\label{hecke}
    Suppose that $\xi \in R$ is invertible and that
    $\bQ=(Q_1,\dots,Q_\ell)\in R^\ell$. The \textbf{cyclotomic Hecke
    algebra} $\HH_n(\xi ,\bQ)=\HH_n^R(\xi ,\bQ)$ of type $G(\ell,1,n)$ and
    with parameters $\xi $ and $\bQ$ is the unital associative $R$-algebra
    with generators $L_1,\dots,L_n,T_1,\dots,T_{n-1}$ and relations
    \begin{align*}
    (L_1-Q_1)\dots(L_1-Q_\ell)&=0,
    & L_rL_t&=L_tL_r,\\
     (T_r+1)(T_r-\xi )&=0,
    &T_rL_r+\delta_{\xi 1}&=L_{r+1}(T_r-\xi +1),\\
    T_sT_{s+1}T_s&=T_{s+1}T_sT_{s+1},\\
    T_rL_t&=L_tT_r,&\text{if }t\ne r,r+1,\\
    T_rT_s&=T_sT_r,&\text{if }|r-s|>1,
  \end{align*}
  where $1\le r<n$, $1\le s<n-1$ and $1\le t\le n$.
\end{Defn}

This definition, which we used in~\cite{HM:GradedCellular}, allows us to simultaneously treat the
degenerate and non-degenerate cyclotomic Hecke algebras of type $G(\ell,1,n)$. By the Morita
equivalence reductions of \cite[Theorem~1.1]{DM:Morita} and \cite[Theorem~5.19]{BK:HigherSchurWeyl} it
is enough to consider the cases where the parameters $\bQ$ are \textbf{integral} in the sense that each
$Q_s$ is an integral power of $\xi $, if $\xi \ne1$, and $\bQ\in\Z^{\ell}$ if $\xi=1$.

Define the \textbf{quantum characteristic} of $\xi $ to be the smallest
positive integer $e$ such that $1+\xi +\dots+\xi ^{e-1}=0$; or $0$ if no such positive integer exists.
Then $e\in\{0,2,3,4,\dots\}$. We fix a
\textbf{multicharge} $\charge=(\kappa_1,\dots,\kappa_\ell)\in\Z^\ell$ such
\begin{enumerate}
\item if $e\ne0$ then $\kappa_l-\kappa_{l+1}\ge n$ for $1\le l<\ell$.
\item if $\xi =1$ then $Q_l\equiv\kappa_l\pmod e$, for $1\le l\le\ell$.
\item if $\xi \ne1$ then $Q_l=\xi ^{\kappa_l}$, for $1\le l\le\ell$.
\end{enumerate}
This choice of multicharge plays a role in what follows only in helping us make a good choice of
modular system as in \cite[\Sect4.2]{HM:GradedCellular}. The multicharge~$\charge$ determines the
parameters $\bQ$. Moreover, for a fixed choice of multicharge, $\H$ depends only on the
quantum characteristic $e$ of~$\xi$, and not on the choice of~$\xi$ itself, by
\cite[Theorem~6.1]{BK:GradedKL}. Therefore, we write $\H=\H(e,\charge)$.

Let $\Sym_n$ be the symmetric group on $\{1,2,\dots,n\}$. Then $\Sym_n$ is
a Coxeter group and $\{s_1,\dots,s_{n-1}\}$ is its standard set of
Coxeter generators, where $s_i=(i,i+1)$ for $1\le i<n$. Let
$\ell\map{\Sym_n}\N$ be the \textbf{length function} on $\Sym_n$ so that
$\ell(w)=k$ if $k$ is minimal such that $w=s_{i_1}\dots s_{i_k}$, for some
$s_{i_j}$ with $1\le i_j<n$. If $w=s_{i_1}\dots s_{i_k}$, with $k=\ell(w)$, then set
$T_w=T_{i_1}\dots T_{i_k}$. Then $T_w$ depends only on $w$, and not on the
choice of \textbf{reduced expression} $w=s_{i_1}\dots s_{i_k}$, because the braid
relations hold in $\H$; see, for example, \cite[Theorem~1.8]{M:Ulect}.

Let $\HSym$ be the $R$-submodule of $\H$ spanned by $\set{T_w|w\in\Sym_n}$.
Then $\HSym$ is isomorphic to the Iwahori-Hecke algebra of $\Sym_n$ with
parameter $\xi$ by \cite[Cor.~3.11]{AK}.

In order to define the bases of $\H$ which underpin this paper we now
review the combinatorics of multipartitions and tableaux.  Recall that a
\textbf{multicomposition} of $n$ is an $\ell$-tuple
$\bmu=(\mu^{(1)},\dots,\mu^{(\ell)})$ of compositions such that
$|\mu^{(1)}|+\dots+|\mu^{(\ell)}|=n$. For each multicomposition let
$\Sym_\bmu=\Sym_{\mu^{(1)}}\times\dots\times\Sym_{\mu^{(\ell)}}$ be the
corresponding \textbf{parabolic}, or Young subgroup, of $\Sym_n$ where we
use the natural embedding $\Sym_\bmu\hookrightarrow\Sym_n$.

A \textbf{multipartition} of $n$ is a multicomposition
$\bmu=(\mu^{(1)},\dots,\mu^{(\ell)})$ of~$n$ such that each component $\mu^{(l)}$ is
a partition, for $1\le l\le\ell$. The set of multicompositions of~$n$ becomes a poset
under the \textbf{dominance order} $\gedom$ where ,if~$\blam$ and~$\bmu$ are multicompositions
of~$n$, then $\blam\gedom\bmu$ if
$$\sum_{k=1}^{l-1}|\lambda^{(k)}|+\sum_{j=1}^i\lambda^{(l)}_j
     \ge\sum_{k=1}^{l-1}|\mu^{(k)}|+\sum_{j=1}^i\mu^{(l)}_j,$$
for $1\le l\le\ell$ and $i\ge1$. If $\blam\gedom\bmu$ and $\blam\ne\bmu$
then we write $\blam\gdom\bmu$. Let $\Parts$ be the poset of the
multipartitions of~$n$ ordered by dominance.

We identify the multipartition $\bmu$ with its \textbf{diagram}
$$[\bmu]=\set{(r,c,l)|1\le c\le \mu^{(l)}_r, r\ge1\text{ and }1\le l\le\ell}.$$
In this way, we will talk of the rows, columns and components of~$\bmu$.
If $\bmu\in\Parts$ let $\bump=({\mu^{(\ell)}}',\dots,{\mu^{(1)}}')$ be the \textbf{conjugate}
multipartition which is obtained from $\bmu$ by reversing the order of its components and then
swapping the rows and columns in each component. We frequently identify $\bmu$ and its diagram
$[\bmu]$, which we think of as an $\ell$-tuple of arrays of boxes in the plane.

Let $\bmu$ be a multicomposition of~$n$. A \textbf{$\bmu$-tableau} is a map
$\t\map{[\bmu]}\{1,2,\dots,n\}$. We think of $\t$ as a labelling of the diagram of $\bmu$ and we define
$\Shape(\t)=\bmu$. A $\bmu$-tableau $\t$ is \textbf{row standard} if $\t(r,c,l)<\t(r,c+1,l)$ whenever
$(r,c,l),(r,c+1,l)\in[\bmu]$. Let $\RStd(\bmu)$ be the set of row standard $\bmu$-tableau. The
\textbf{conjugate} of the~$\bmu$-tableau $\t$ is the $\bump$-tableau which is obtained from~$\t$ by
reversing its components and then swapping its rows and columns in each component.

Suppose now that $\bmu$ is a multipartition. Then a $\bmu$-tableau $\t$ is
\textbf{standard} if $\t$ and $\t'$ are both row standard tableaux. Let
$\Std(\bmu)$ be the set of standard $\bmu$-tableaux and
$\SStd(\bmu)=\set{(\s,\t)|\s,\t\in\Std(\bmu)}$ be the set of pairs of
standard $\bmu$-tableaux. For convenience we set
$$\Std(\Parts)=\bigcup_{\bmu\in\Parts}\Std(\bmu)\quad\text{and}\quad
\SStd(\Parts)=\bigcup_{\bmu\in\Parts}\SStd(\bmu).$$

Suppose that $\s$ is a row standard $\blam$-tableau and that $\t$ is a row
standard $\bmu$-tableau, for multicompositions $\blam$ and $\bmu$ of~$n$.
For each non-negative integer $m$ define $\s_m$ and $\t_m$ to be the
subtableaux of $\s$ and $\t$, respectively, which contain
$\{1,2,\dots,m\}$. Then $\s$ \textbf{dominates}~$\t$, and we write
$\s\gedom\t$, if
$$\Shape(\s_m)\gedom\Shape(\t_m), \qquad\text{ for } 1\le m\le n.$$
It is straightforward to check that $\s\gedom\t$ if and only if
$\t'\gedom\s'$. Observe also that $\blam\gedom\bmu$ if and only if $\tlam\gedom\tmu$.

We extend the dominance order to $\SStd(\Parts)$ in two ways by declaring
that if $(\s,\t)\in\SStd(\blam)$ and $(\u,\v)\in\SStd(\bmu)$ then
\begin{align*}
(\s,\t)\gedom(\u,\v)&\quad\text{if }\blam\gdom\bmu\text{ or }
     \blam=\bmu\text{ and } \s\gedom\u\text{ and }\t\gedom\v,\\
(\s,\t)\Gedom(\u,\v)&\quad\text{if }\s\gedom\u\text{ and } \t\gedom\v.
\end{align*}
By definition, $(\s,\t)\Gedom(\u,\v)$ implies that $(\s,\t)\gedom(\u,\v)$,
but the converse is false in general. As above, we write
$(\s,\t)\Gdom(\u,\v)$, and $(\s,\t)\gdom(\u,\v)$, if $(\s,\t)\ne(\u,\v)$
and $(\s,\t)\Gedom(\u,\v)$ and $(\s,\t)\gedom(\u,\v)$, respectively. The
partial order $\Gedom$ is the \textbf{strong dominance} order on
$\SStd(\Parts)$.

Suppose that $\bmu$ is a multicomposition of~$n$ and define $\tmu$ to be the unique row standard
$\bmu$-tableau such that $\tmu\gedom\s$ whenever $\s$ is a row standard
$\bmu$-tableau. That is,
$\tmu$ is the $\bmu$-tableau with the numbers $1,2,\dots,n$ entered in
order from left to right along the rows of each component of $\tmu$. Next
if $\s\in\Std(\bmu)$ let $d(\s)$ be the unique permutation in~$\Sym_n$
such that $\s=\tmu d(\s)$.

We are now ready to define the first bases of $\H$ that we will need. Let
$*\map{\H}\H$ be the unique anti-isomorphism of $\H$ which fixes each of
the generators in Definition~\ref{hecke}.

Let $\bmu$ be a multicomposition. Following~\cite{DJM:cyc,M:gendeg}, set
\begin{xalignat*}{3}
  u^+_\bmu&=\prod_{k=2}^l\prod_{m=1}^{|\mu^{(1)}|+\dots+|\mu^{(k-1)}|}\!\!(L_m-Q_k),&
  x_\bmu&= \sum_{w\in\Sym_\bmu} T_w,\\
  u^-_\bmu&=\prod_{k=1}^{\ell-1}\prod_{m=1}^{|\mu^{(1)}|+\dots+|\mu^{(\ell-k+1)}|}\!\!(L_m-Q_k),&
  y_\bmu&=\sum_{w\in\Sym_\bmu} (-\xi )^{-\ell(w)}T_w,
\end{xalignat*}
and let $m_\bmu=u^+_\bmu x_\bmu$ and $n_\bmu=u^-_\bmu y_\bmu$. Finally, if $\bmu$ is 
a multipartition define
$$m_{\s\t}=T_{d(\s)}^*m_\bmu T_{d(\t)} \quad\text{and}\quad
  n_{\s\t}=T_{d(\s)}^*n_\bmu T_{d(\t)},$$
for $(\s,\t)\in\SStd(\Parts)$. It follows easily from the relations
in~$\H$ (see \cite[Remark~3.7]{DJM:cyc}), that the elements $x_\bmu$
and $u_\bmu^+$, and $y_\bmu$ and $u_\bmu^-$, commute, so that
$m_{\s\t}^*=m_{\t\s}$ and $n_{\s\t}^{*}=n_{\t\s}$. We have the
following important and well-known result.

\begin{Theorem}\label{bases} Suppose that $R$ is an integral domain. Then:
\begin{enumerate}
\item $\set{m_{\s\t}|(\s,\t)\in\SStd(\Parts)}$ is a cellular basis of $\H$.
\item $\set{n_{\s\t}|(\s,\t)\in\SStd(\Parts)}$ is a cellular basis of $\H$.
\end{enumerate}
\end{Theorem}

Part~(a) is proved in \cite[Theorem~3.26]{DJM:cyc} and \cite[Theorem~6.3]{AMR} for the non-degenerate
($\xi \ne1)$ and degenerate ($\xi =1$) cases, respectively.  Part~(b) is can be proved in the same way
or arguing by specialization from the case where $\H$ is defined over a `generic' ground ring in which
case $\H$ has a $\Z$-linear automorphism which interchanges these two bases; see, for example,
\cite[(3.1)]{M:gendeg}.

\subsection{Seminormal forms and strong dominance}
In this subsection we give a necessary condition for the product $m_{\s\t}n_{\u\v}$ to be non-zero. To
prove this we show that the transition matrices between the standard and seminormal bases of $\H$ are
ordered by strong dominance, a theme that continues throughout this paper.

Suppose that $(\s,\t)\in\SStd(\Parts)$ and $1\le k\le n$. If $k$ appears in
row~$r$ and column~$c$ of~$\t^{(l)}$ then define
$$\cont_\t(k)=\begin{cases}
                 \xi^{c-r}Q_l,&\text{if }\xi\ne1,\\
                  c-r+Q_l,&\text{if }\xi=1.
\end{cases}$$
Then by \cite[Prop.~3.7]{JM:cyc-Schaper} and \cite[Lemma~6.6]{AMR},
corresponding to the cases $\xi\ne1$ and $\xi=1$, respectively,
\begin{align}\label{Murphy}
m_{\s\t}L_k=\cont_\t(k) m_{\s\t}+
                \sum_{\substack{(\u,\v)\in\SStd(\Parts)\\(\u,\v)\gdom(\s,\t)}}r_{\u\v}m_{\u\v},
\end{align}
for some $r_{\u\v}\in R$.

For the rest of this subsection we assume that $R=\K$ is a field and that
$\H=\HK$ is semisimple. Equivalently, we assume that if $\s,\t\in\Std(\Parts)$
then $\cont_\s(k)=\cont_\t(k)$, for all $1\le k\le n$, if and only if $\s=\t$.

The following definition has its origins in the work of
Murphy~\cite{M:Nak}.

\begin{Defn}[\protect{\cite[Defn~3.1]{M:seminormal}}]\label{seminormal
basis}
    Suppose that $\blam\in\Parts$ and $(\s,\t)\in\SStd(\blam)$. Define
    $$F_\t=\prod_{k=1}^n\prod_{\substack{\s\in\Std(\Parts)\\
             \cont_\s(k)\ne\cont_\t(k)}}
    \frac{L_k-\cont_\s(k)}{\cont_\t(k)-\cont_\s(k)}\in\HK.$$
    Set $f_{\s\t}=F_\s m_{\s\t} F_\t$.
\end{Defn}

By (\ref{Murphy}),
$f_{\s\t}=m_{\s\t}+\sum_{(\u,\v)\gdom(\s,\t)}r_{\u\v}m_{\u\v}$, for some
$r_{\u\v}\in\K$. In particular, applying Theorem~\ref{bases} shows that
$\{f_{\s\t}\}$ is a basis of~$\HK$. Moreover, if $(\s,\t)\in\SStd(\Parts)$
and $1\le k\le n$ then
\begin{equation}\label{eigenvectors}
f_{\u\v}L_k=\cont_\v(k)f_{\u\v}
\end{equation}
by \cite[Prop.~2.6(iii)]{M:gendeg} and \cite[Page~109]{AMR}.

If $\s$ is a tableau and $1\le k\le n$ define $\comp_\s(k)=c$ if $k$ appears in~$\s^{(c)}$.
If~$\s$ and~$\t$ are tableau we write $\comp(\s)\le\comp(\t)$ if $\comp_\s(k)\le\comp_\t(k)$ for
$1\le k\le n$. Then $\comp(\s)\le\comp(\t)$ whenever $\s\gedom\t$ , but the converse is false
in general.

The following result is well-known. We include a proof because we do not know of a reference for
it.

\begin{Lemma} \label{lm22} Suppose that $\blam, \bmu\in\Parts$ and $\s\in\Std(\blam)$ and
that there exist two integers $a<b$ which are in the same row of $\tmu$ and in the same column
of~$\s$. Then there exists an element $w\in\Sym_{\{a,a+1,\cdots,b\}}$ and an integer $a\leq c<b$
such that
\begin{enumerate}
\item $\s w$ is standard; and
\item $\ell\bigl(d(\s)w\bigr)=\ell\bigl(d(\s)\bigr)+
\ell(w)$; and
\item $c,c+1$ are in the same row of $\tmu$ and the same
column of $\s w$.
\end{enumerate}
\end{Lemma}

\begin{proof} By assumption, the integers $a, a+1,\cdots, b$ are all in the same row of $\tmu$ so
because~$\s$ is standard we may assume, without loss of generality, that $a$ appears in row~$r$
of~$\s^{(l)}$ and that $b$ appears in row~$r+1$ of~$\s^{(l)}$, where $1\le l\le\ell$. We now argue
by induction on $b-a$.

If $b-a=1$ then there is nothing to prove, so suppose that $b-a>1$. Let
$c<b$ be maximal such that
\begin{enumerate}
  \item $\comp_\s(c)<l$ or $\comp_\s(c)=l$ and $c$ appears in the first $r$ rows of $\s^{(l)}$, and
  \item  $\comp_s(c+1)>l$ or $\comp_\s(c)=l$ and $c+1$ appears below row~$r$ of $\s^{(l)}$.
\end{enumerate}
In particular, this means that the numbers $c+1,c+2,\cdots,b$ all appear `below' row~$r$ of $\s^{(l)}$.
Let $\t=\s w$, where $w=(c,c+1,\dots,b)=s_{b-1}\dots s_c\in\Sym_n$. Then $\s\gedom\t$ and
$\ell(d(\t))=\ell(d(\s))+\ell(w)$.  The integers $a<c$ are in the same row of $\t^{\bmu}$ and in the
same column of $\t=\s w$. Note that $c-a<b-a$. The Lemma now follows by induction.
\end{proof}

\begin{Prop}\label{multicompositions}
  Suppose that $R=\K$ and $\blam$ is a multicomposition of~$n>0$. Then there exist scalars
  $a^\blam_{\u\v}\in\K$ such that
  $$m_\blam = \sum_{(\u,\v)\in\SStd(\Parts)} a^\blam_{\u\v}f_{\u\v},$$
  and $a^\blam_{\s\t}\ne0$ only if $\comp(\tlam)\ge\comp(\u)$, $\comp(\tlam)\ge\comp(\v)$  and $i$
  and $j$ are in different columns of $\u$ and $\v$ whenever they are in the same column of
  $\tlam$.
\end{Prop}

\begin{proof}
  We consider only the case when $\xi\ne1$. The case $\xi =1$ is similar and may be proved using
  the results of \cite[\Sect6]{AMR}. The only real difference between the cases $\xi\ne1$ and
  $\xi=1$ is the choice of content function: if $\xi\ne1$ then $\cont_\v(k)=\xi^{c-r}Q_l$, when
  $\v(r,c,l)=k$, and if $\xi=1$ then, instead, $\cont_\v(k)=c-r+Q_l$. Analogous minor
  `logarithmic' adjustments are required in the argument below when $\xi=1$.

  By (\ref{Murphy}) there exist $a^\blam_{\u\v}\in\K$ such that we can write
  $$m_\blam = f_{\tlam\tlam}+\sum_{(\u,\v)\gedom(\tlam,\tlam)} a^\blam_{\u\v}f_{\u\v}.$$
  Suppose that $s_i=(i,i+1)\in\Sym_\blam$, where $1\le i<n$. It is well-known and
  easy to check that $x_\blam T_i=\xi x_\blam$ (see, for example,
  \cite[Lemma~3.2]{M:Ulect}), so that $m_\blam T_i=\xi m_\blam$. To compute
  the action of $T_i$ on the right hand side of the last equation we need to recall
  how $\H$ acts on its seminormal basis.

  By \cite[Prop.~2.7]{M:gendeg}, $f_{\u\v}T_i=\xi f_{\u\v}$ if $i$ and
  $i+1$ are in the same row of $\v$ and $f_{\u\v}T_i=-f_{\u\v}$ if $i$ and
  $i+1$ are in the same column of $\v$. Otherwise, $i$ and $i+1$ are in
  different rows and columns of $\v$ so we set $\t=\v(i,i+1)$ and
  $c_\v=\cont_\v(i)$ and $c_\t=\cont_{\t}(i)=\cont_\v(i+1)$. Without loss of
  generality, $\v\gdom\t$. Hence, by \cite[Prop.~2.7]{M:gendeg} we have
  \begin{align*}
    f_{\u\v}T_i &= \frac{(\xi -1)c_\t}{(c_\t-c_\v)}f_{\u\v} +f_{\u\t},\\
    f_{\u\t}T_i&= \frac{(\xi c_\v-c_\t)(c_\v-\xi c_\t)}{(c_\v-c_\t)^2}f_{\u\v}
               + \frac{(\xi -1)c_\v}{(c_\v-c_\t)}f_{\u\t}.
  \end{align*}
  Therefore, $a^\blam_{\u\v}=0$ if $i$ and $i+1$ are in the same column. If
  $\v$ and $\t$ are both standard tableau and $\v\gdom\t$ then
  \begin{align*}
    \xi a^\blam_{\u\v}&=\frac{(\xi -1)c_\t}{(c_\t-c_\v)}a^\blam_{\u\v}
                          +\frac{(\xi c_\v-c_\t)(c_\v-\xi c_\t)}{(c_\v-c_\t)^2}a^\blam_{\u\t}\\
    \xi a^\blam_{\u\t}&=a^\bmu_{\u\v}+\frac{(\xi -1)c_\v}{(c_\v-c_\t)}a^\blam_{\u\t}.
  \end{align*}
  Solving these equations shows that $a^\blam_{\u\v}=(c_\v-\xi c_\t)/(c_\v-c_\t)\cdot
  a^\blam_{\u\t}$.  In particular, $a^\blam_{\u\v}\ne0$ if and only if $a^\blam_{\u\t}\ne0$, since
  $c_\v\ne c_\t$ and $c_\v\ne \xi c_\t$ (because $\HK$ is semisimple), whenever~$i$ and $i+1$ are in the same
  row of~$\tlam$ and in different columns of~$\v$. Applying Lemma~\ref{lm22} and acting by~$\Sym_\blam$
  now shows that $a^\blam_{\u\v}=0$ whenever there exist $i$ and $j$ which are in the same column
  of~$\v$ and the same row of~$\tlam$.

  To complete the proof we need to show that $\comp(\tlam)\ge\comp(\u)$ and
  $\comp(\tlam)\ge\comp(\v)$.  In fact, since $m_\blam=m_\blam^*$ it is enough to show that
  $\comp(\tlam)\ge\comp(\v)$. Let $1\le l\le\ell$ be minimal such that $|\lambda^{(l)}|>0$. If $l=\ell$
  then $\comp(\tlam)\ge\comp(\v)$ for any tableau $\v$ so in this case there is nothing to prove.
  Suppose then that $1\le l<\ell$ and fix $(\s,\t)\in\Std(\Parts)$ with $a^\blam_{\s\t}\ne0$.
  Let~$\bnu$ be the multicomposition with $\nu^{(k)}=\lambda^{(k)}$ if $k\ne l,l+1$, $\nu^{(l)}=(0)$
  and where $\nu^{(l+1)}$ is the composition obtained by concatenating $\lambda^{(l)}$ and
  $\lambda^{(l+1)}$. Let $m=|\lambda^{(l)}|$.
  Then $x_\blam=x_\bnu$ and
  $$m_\blam = m_\bnu \prod_{k=1}^m(L_k-Q_{l+1})
       =\sum_{(\u,\v)\in\SStd(\Parts)}a^\bnu_{\u\v}f_{\u\v}\prod_{k=1}^m(L_k-Q_{l+1}).$$
  Thus, $a^\blam_{\u\v}=a^\bnu_{\u\v}(\cont_\v(1)-Q_{l+1})\dots(\cont_\v(m)-Q_{l+1})$ by
  (\ref{eigenvectors}). By induction
  $a^\bnu_{\u\v}\ne0$ only if $\comp(\tnu)\ge\comp(\v)$. In particular, if
  $a^\bnu_{\u\v}\ne0$ then $\comp_\v(k)\le\comp_{\tlam}(k)=\comp_{\tnu}(k)$ whenever $m<k\le n$.
  Consequently, if $\comp(\tlam)\not\ge\comp(\v)$ then there must exist a $k$ with $1\le k\le m$ and
  $\comp_\v(k)=l+1$. Therefore, since $\v$ is standard and $\comp_{\v}(k')\leq\comp_{\tnu}(k')$, there
  exists $1\le k'\le k$ in the first row and column of $\v^{(l+1)}$ so that $\cont_\v(k')=Q_{l+1}$.
  Consequently, $a^\blam_{\u\v}=0$ and $f_{\u\v}$ does not appear in $m_{\blam}$ in this case. It
  follows that $a^\blam_{\u\v}\ne0$ only if $\comp(\tlam)\ge\comp(\v)$ as required.
\end{proof}

\begin{Cor}\label{multipartitions}
  Suppose that $\bmu\in\Parts$ is a multipartition of $n$. Then there exist
  scalars $a_{\u\v}\in\K$ such that
  $$m_\bmu=\sum_{\substack{(\u,\v)\in\SStd(\Parts)\\(\u,\v)\Gedom(\tmu,\tmu)}}
            a_{\u\v}f_{\u\v}.$$
\end{Cor}

\begin{proof}
  Suppose that $a_{\u\v}\ne0$ for some $(\u,\v)\in\SStd(\Parts)$.  By
  Proposition~\ref{multicompositions} $a_{\u\v}\ne0$ only if $\comp(\v)\le\comp(\tmu)$ and $i$
  and $j$ are in different columns of $\v$ whenever they are in the same row of $\tmu$.
  These two conditions imply that if $1\le m\le n$ then the node containing~$m$ in~$\v$ is
  never below the node which contains~$m$ in~$\tmu$. This implies that $\v\gedom\tmu$.
  Similarly, or by applying the involution~$*$, $\u\gedom\tmu$.
\end{proof}

We want to generalize result to an arbitrary basis element $m_{\s\t}$. To do this we use the following
two technical results, the first of which is due to Murphy~\cite{Murphy:basis}.  First, recall from
\cite[Theorem~3.8]{M:Ulect} that if $\hat\v$ and $\v$ are two row standard $\bnu$-tableaux then
$\v\gedom\hat\v$ if and only if $d(\v)\le d(\hat\v)$ where~$\le$ is the Bruhat order on~$\Sym_n$. That
is, $u\le v$ if $u$ has a reduced expression which is a subexpression of some reduced expression
of~$v$.

\begin{Lemma}[\protect{Murphy~\cite[Lemma~3.5]{Murphy:basis}}]\label{technical}
   Suppose that $\hat\v\in\RStd(\bmu)$, $\s\in\Std(\bmu)$ and $u,w\in\Sym_n$
   such that $u\le w$, $\hat{\v}\gedom\s$, $\hat\v u$ is row standard,  $\s\gdom \s w$ and
   $\ell(d(\s)w)=\ell(d(\s))+\ell(w)$. Then $\hat\v u\gedom\s w$.
\end{Lemma}

When comparing our statement of Lemma~\ref{technical} with Murphy's result note that Murphy
considers row standard $\mu$-tableaux, where $\mu$ is a composition. Let
$\mu^\vee=\mu^{(1)}\vee\dots\vee\mu^{(\ell)}$ be the composition obtained by concatenating the
parts of~$\bmu$. Then any row standard $\bmu$-tableau~$\t$ corresponds to the row standard
$\mu^\vee$-tableau $\t^\vee=\t^{(1)}\vee\dots\vee\t^{(\ell)}$ obtained by concatenating the rows
of $\t$.  This observation allows us to rewrite Murphy's lemma in the form above.

\begin{Lemma}\label{technical2}
  Suppose that $\u,\v\in\Std(\bnu)$ and $\t\in\Std(\bmu)$ are standard
  tableaux such that $\v\gedom\tmu$ and that $x\in\Sym_n$ with $x\leq d(\t)$. Then
  $m_{\u\v}T_x$ can be written as a linear combination of terms
  $m_{\c\d}$ such that $\c\gedom\u$ and $\d\gedom\t$.
\end{Lemma}

\begin{proof} Observe that $x_\bnu\HSym$ is the permutation module for $\HSym$, in the sense of
\cite[Chapt.~3]{M:Ulect}, which is indexed by the composition $\nu^\vee$ which is obtained by
concatenating the components of $\bnu$.  By
\cite[Cor.~3.4]{M:Ulect}, $\set{x_\bnu T_{d(\b)}|\b\in\RStd(\bnu)}$is a
basis of the $\HSym$-module $x_\bnu\HSym$.  Moreover, the same
result shows that if $1\le i<n$ then
\begin{equation}\label{permutation}
x_\bnu T_{d(\v)}T_i=\begin{cases}
    \xi x_\bnu T_{d(\v)},&\text{if }\v s_i\notin\RStd(\bnu),\\
    \phantom{\xi}x_\bnu T_{d(\v)s_i},&\text{if }\v\gdom\v s_i\in\RStd(\bnu),\\
    \xi x_\bnu T_{d(\v)s_i}+(\xi-1) x_\bnu T_{d(\v)},&\text{if }\v\ldom\v s_i\in\RStd(\bnu).
\end{cases}
\end{equation}
We note that if $\v s_i\in\RStd(\bnu)$ then either $\v\gdom\v s_i$ or $\v s_i\gdom\v$.  Applying
(\ref{permutation}) recursively, we see that $x_\bnu T_{d(\v)}T_x$ is a linear combination of terms of
the form $x_\bnu T_{d(\b)}$ where $\b=\tnu d(\hat\v)u\in\RStd(\bnu)$, $\hat\v$ is a row standard
$\bnu$-tableau such that $\hat\v\gedom\v$, $u\leq x\leq d(\t)$ and
$\ell(d(\hat\v)u)=\ell(d(\hat\v))+\ell(u)$. Thus, we have $\hat\v\in\RStd(\bnu)$,
$\b=\hat{\v}u\in\RStd(\bnu)$ and $\hat\v\gedom\v\gedom\tmu$. Hence, setting $\s=\tmu$ and $w=d(\t)$ we
see that all of the conditions in Murphy's Lemma~\ref{technical} are satisfied so that
$\b=\hat\v u\gedom\s w=\t$. That is, $x_\bnu T_{d(\v)}T_{x}$ can be written as a linear combination of
$x_\bnu T_{d(\b)}$ where $\b\in\RStd(\bnu)$ and $\b\gedom\t$.

We have shown that $m_{\u\v}T_x$ can we written as a linear combination of
terms of the form $m_{\u\b}$, where $\b\gedom\t$ is row standard. Hence, by
\cite[Prop.~3.18]{DJM:cyc} we can write $m_{\u\v}T_x$ as a linear
combination of elements of the form $m_{\c\d}$ where
$\c\gedom\u$ and $\d\gedom\t$. (Note that the standard tableaux $\c$
and $\d$ do not necessarily have shape $\bmu$.) This completes the proof.
\end{proof}

\begin{Theorem}\label{strong}
Suppose that $R=\K$ and $(\s,\t)\in\SStd(\Parts)$. Then:
\begin{enumerate}
\item There exist scalars $a_{\u\v}\in\K$ such that
$$m_{\s\t}=f_{\s\t}+\Sum_{\substack{(\u,\v)\in\SStd(\Parts)\\(\u,\v)\Gdom(\s,\t)}}a_{\u\v}\,f_{\u\v}.$$
\item There exist scalars $b_{\u\v}\in\K$ such that
 $$f_{\s\t}=m_{\s\t}+\Sum_{\substack{(\u,\v)\in\SStd(\Parts)\\(\u,\v)\Gdom(\s,\t)}}b_{\u\v}\,m_{\u\v}.$$
\end{enumerate}
\end{Theorem}

\begin{proof} We first prove part~(b) using induction on~$\Gedom$. If $\bEta=((n),(0),\dots,(0))$
then, directly from the definitions, $f_{\teta\teta}=m_{\teta\teta}=m_\bEta$. Hence,~(b) is
automatically true in this case.  Notice also that $(\teta,\teta)\Gedom(\s,\t)$, for all
$(\s,\t)\in\SStd(\Parts)$.  Suppose now that $(\s,\t)\in\SStd(\bmu)$ and $(\s,\t)\ne(\teta,\teta)$.
Then, by Corollary~\ref{multipartitions} and induction, there exist scalars $a_{\u\v}\in\K$ such that
$$f_{\tmu\tmu}=m_\mu+\sum_{(\u,\v)\Gdom(\tmu,\tmu)}a_{\u\v}m_{\u\v}.$$ Suppose that
$(\s,\t)\in\SStd(\bmu)$. By
\cite[Proposition~4.1 and Lemma~4.3]{M:gendeg}, there exists elements
$\Phi_\s,\Phi_\t\in\HSym$ such that
$f_{\s\t}=\Phi_\s^*f_{\tmu\tmu}\Phi_\t$. (In \cite{M:gendeg} this is proved
only in the case when $\xi \ne1$. The case when $\xi =1$ follows by exactly the
same argument.) Therefore, by the last displayed equation,
$$f_{\s\t}=\Phi^*_\s f_{\tmu\tmu}\Phi_\t
          =\sum_{(\u,\v)\Gedom(\tmu,\tmu)}a_{\u\v}\Phi^*_\s m_{\u\v}\Phi_\t,$$
where for convenience we set $a_{\tmu\tmu}=1$.  By the argument
of~\cite[Proposition~4.1(ii)]{M:gendeg}, $\Phi_\t=\sum_{b\le d(\t)}p_{\t b}T_b$, for some
$p_{\t b}\in\K$ where, in the sum, $b\in\Sym_n$ (with $p_{\t d(\t)}=1$). By
Lemma~\ref{technical2} we can write $m_{\u\v}T_b$ as a linear combination of elements of the
form $m_{\c\d}$ with $\c\gedom\u$ and $\d\gedom\t$. Hence, we can write $m_{\u\v}\Phi_\t$ as a
linear combination of terms $m_{\c\d}$ with $(\c,\d)\Gedom(\tmu,\t)$. Applying the
left handed version of Lemma~\ref{technical2} to each of the terms $\Phi_\s^* m_{\c\d}$, we
see that each $\Phi_\s^* m_{\u\v}\Phi_\t$ can be written as a linear combination of elements
of the form $m_{\a\b}$ with $(\a,\b)\Gedom(\s,\t)$.  Hence, $f_{\s\t}$ can be written as a
linear combination of elements $m_{\a\b}$ with $(\a,\b)\Gedom(\s,\t)$ giving~(b). Inverting
the equations in~(b) gives~(a), completing the proof.
\end{proof}

We can now prove the promised criterion for the product $m_{\s\t}n_{\u\v}$ to be
non-zero.  Notice that unlike Theorem~\ref{strong}, which requires $R=\K$, the next
results are valid over an arbitrary integral domain.

\begin{Cor}\label{tilting}
Suppose that $R$ is an integral domain and that $(\s,\t),(\u,\v)\in\SStd(\Parts)$. Then:
\begin{enumerate}
\item  $m_{\s\t}n_{\u\v}\ne0$ only if $\u'\gedom\t$, and,
\item  $n_{\u\v}m_{\s\t}\ne0$ only if $\v'\gedom\s$.
\end{enumerate}
\end{Cor}

\begin{proof}
  Parts~(a) and (b) are equivalent by applying the anti-isomorphism~$*$ of $\H$ which fixes each
  generator, so we prove only~(a).  Let $K$ be the field of fractions of $R$. Then by embedding $\H^R$
  into $\H^K$ and choosing a suitable modular system $(\O,\K,K)$ (see \cite[\Sect4.2]{HM:GradedCellular}
  for example), we can reduce to the case where $R=\K$.  Following \cite{M:gendeg}, if
  $(\u,\v)\in\SStd(\Parts)$ then define $g_{\u\v}=F_{\u'}n_{\u\v} F_{\v'}$. Then by repeating the
  arguments of Theorem~\ref{strong} (or by applying a suitable automorphism in the generic case as in
  the proof of \cite[Prop.~3.4]{M:gendeg}), it follows that
  $$n_{\u\v} = g_{\u\v}+\sum_{(\c,\d)\Gdom(\u,\v)}d_{\c\d}g_{\c\d},$$
  for some scalars $d_{\c\d}\in\K$. Hence, by Theorem~\ref{strong} we
  have
  $$0\ne m_{\s\t}n_{\u\v}=\Big(f_{\s\t}+\sum_{(\a,\b)\Gdom(\s,\t)}c_{\a\b}f_{\a\b}\Big)
      \Big(g_{\u\v}+\sum_{(\c,\d)\Gdom(\u,\v)}d_{\c\d}g_{\c\d}\Big).$$
  Hence,  there exist tableaux $(\a,\b)\Gedom(\s,\t)$ and
  $(\c,\d)\Gedom(\u,\v)$ such that $f_{\a\b}g_{\c\d}\ne0$. By
  \cite[Corollary~3.8]{M:gendeg}, $f_{\a\b}g_{\c\d}\ne0$ only if $\c'=\b.$
  Therefore, $\u'\gedom\c'=\b\gedom\t$, so that $\u'\gedom\t$ as required.
\end{proof}

\begin{Remark}
Both the statement and proof of part~(b) of Corollary~\ref{tilting} is
essentially the same as \cite[Lemma~5.4]{M:tilting}. Unfortunately,
in~\cite{M:tilting} the second author confused the two partial orders~$\gedom$
and~$\Gedom$ on~$\SStd(\Parts)$, so all that was actually proved in that
paper was that $\Shape(\u')\gedom\Shape(\t)$ whenever
$m_{\s\t}n_{\u\v}\ne0$. As a consequence, the current paper completes the proof of
\cite[Lemma 5.8]{M:tilting} which requires the full strength of
Corollary~\ref{tilting} (and Lemma \ref{lm22}).
\end{Remark}

\begin{Cor}\label{L_k action}
Suppose that $R$ is an integral domain, $1\le k\le n$ and $(\s,\t)\in\SStd(\Parts)$. Then there exist
scalars $c_{\u\v}\in R$ such that
$$m_{\s\t}L_k=\cont_\t(k)m_{\s\t}+\sum_{\substack{(\u,\v)\in\SStd(\Parts)\\(\u,\v)\Gdom(\s,\t)}}
                 c_{\u\v}m_{\u\v}.$$
\end{Cor}

\begin{proof}
  As in the proof of Corollary~\ref{tilting} it is enough to consider the case when $R=\K$ and $\HK$ is
  semisimple. Using parts~(a) and~(b) of Theorem~\ref{strong}, to switch between the standard and
  seminormal bases, together with (\ref{eigenvectors}) for the second equality we see that
\begin{align*} m_{\s\t}L_k
   &=\Big(f_{\s\t}+\sum_{\substack{(\u,\v)\in\SStd(\Parts)\\(\u,\v)\Gdom(\s,\t)}}a_{\u\v}f_{\u\v}\Big)L_k\\
   &=\cont_\t(k)f_{\s\t}+\sum_{\substack{(\u,\v)\in\SStd(\Parts)\\(\u,\v)\Gdom(\s,\t)}}
               \cont_\v(k)a_{\u\v}\,f_{\u\v}\\
   &=\cont_\t(k)m_{\s\t}+\sum_{\substack{(\u,\v)\in\SStd(\Parts)\\(\u,\v)\Gdom(\s,\t)}}
               \cont_\v(k)c_{\u\v}\,m_{\u\v},\\
\end{align*}
for some $c_{\u\v}\in\K$. This completes the proof.
\end{proof}

We remark that it is not hard to see that Corollary~\ref{multipartitions},
Corollary~\ref{L_k action}  and the two statements in Theorem~\ref{strong} are all,
in fact, equivalent. There are analogous (equivalent) statements for the basis
$\{n_{\s\t}\}$. We leave the details to the interested reader.

Similarly, one can show that  $n_{\s\t}L_k=\cont_{\t'}(k)n_{\s\t}$ plus a linear combination of terms
$n_{\u\v}$ with $(\u,\v)\Gdom(\s,\t)$.

\section{Cyclotomic Quiver Hecke algebras and graded Specht modules} We now switch to the
cyclotomic quiver Hecke algebras, which were introduced in a series of papers by
Khovanov-Lauda~\cite{KhovLaud:diagI}, Rouquier~\cite{Rouq:2KM} and Brundan and
Kleshchev~\cite{BK:GradedKL}. These algebras are certain naturally graded algebras which depend on
$e$. As we recall, when they are defined over suitable fields, they are isomorphic to the cyclotomic
Hecke algebras of the last section if we take $e$ to be equal to the quantum characteristic of the
parameter $\xi$ and choose appropriate parameters.

\subsection{Cyclotomic Quiver Hecke algebras and graded induction}\label{grading}
Recall from the introduction that we have fixed an integer $e\in\{0,2,3,4,\dots\}$ and that
$I=\Z/e\Z$. Let $\Gamma$ be the oriented quiver with vertex set~$I$ and directed edges
$i\longrightarrow i+1$, for $i\in I$. To the quiver $\Gamma$ we attach the standard Lie
theoretic data of a Cartan matrix $(a_{ij})_{i,j\in I}$, fundamental weights
$\set{\Lambda_i|i\in I}$, positive weights $X^+=\sum_{i\in I}\N\Lambda_i$,  positive roots
$Q_+=\bigoplus_{i\in I}\N\alpha_i$ and we let $(\cdot,\cdot)$ be the bilinear form determined by
$$(\alpha_i,\alpha_j)=a_{ij}\qquad\text{and}\qquad
          (\Lambda_i,\alpha_j)=\delta_{ij},\qquad\text{for }i,j\in I.$$
More details can be found, for example, in ~\cite[Chapt.~1]{Kac}.

\begin{Defn}\label{quiver Hecke relations}
  The \textbf{cyclotomic quiver Hecke algebra}, or \textbf{cyclotomic Khovanov-Lauda--Rouquier algebra},
$\R$ of weight $\Lambda$ and type $\Gamma_e$ is the unital associative $R$-algebra with generators
  $$\{\psi_1,\dots,\psi_{n-1}\} \cup
         \{ y_1,\dots,y_n \} \cup \set{e(\bi)|\bi\in I^n}$$
  and relations
\begin{align*}
y_1^{(\Lambda,\alpha_{i_1})}e(\bi)&=0,
& e(\bi) e(\bj) &= \delta_{\bi\bj} e(\bi),
&{\textstyle\sum_{\bi \in I^n}} e(\bi)&= 1,\\
y_r e(\bi) &= e(\bi) y_r,
&\psi_r e(\bi)&= e(s_r{\cdot}\bi) \psi_r,
&y_r y_s &= y_s y_r,\\
\end{align*}
\vskip-38pt
\begin{align*}
\psi_r y_s  &= y_s \psi_r,&\text{if }s \neq r,r+1,\\
\psi_r \psi_s &= \psi_s \psi_r,&\text{if }|r-s|>1,\\
\end{align*}
\vskip-30pt
\begin{align*}
  \psi_r y_{r+1} e(\bi) &= \begin{cases}
      (y_r\psi_r+1)e(\bi),\hspace*{18mm} &\text{if $i_r=i_{r+1}$},\\
    y_r\psi_r e(\bi),&\text{if $i_r\neq i_{r+1}$}
  \end{cases} \\
  y_{r+1} \psi_re(\bi) &= \begin{cases}
      (\psi_r y_r+1) e(\bi),\hspace*{18mm} &\text{if $i_r=i_{r+1}$},\\
    \psi_r y_r e(\bi), &\text{if $i_r\neq i_{r+1}$}
  \end{cases}\\
  \psi_r^2e(\bi) &= \begin{cases}
       0,&\text{if $i_r = i_{r+1}$},\\
      e(\bi),&\text{if $i_r \ne i_{r+1}\pm1$},\\
      (y_{r+1}-y_r)e(\bi),&\text{if  $e\ne2$ and $i_{r+1}=i_r+1$},\\
       (y_r - y_{r+1})e(\bi),&\text{if $e\ne2$ and  $i_{r+1}=i_r-1$},\\
      (y_{r+1} - y_{r})(y_{r}-y_{r+1}) e(\bi),&\text{if $e=2$ and
$i_{r+1}=i_r+1$}
\end{cases}\\
\psi_{r}\psi_{r+1} \psi_{r} e(\bi) &= \begin{cases}
    (\psi_{r+1} \psi_{r} \psi_{r+1} +1)e(\bi),\hspace*{7mm}
       &\text{if $e\ne2$ and $i_{r+2}=i_r=i_{r+1}-1$},\\
  (\psi_{r+1} \psi_{r} \psi_{r+1} -1)e(\bi),
       &\text{if $e\ne2$ and $i_{r+2}=i_r=i_{r+1}+1$},\\
  \big(\psi_{r+1} \psi_{r} \psi_{r+1} +y_r\\
  \qquad -2y_{r+1}+y_{r+2}\big)e(\bi),
    &\text{if $e=2$ and $i_{r+2}=i_r=i_{r+1}+1$},\\
  \psi_{r+1} \psi_{r} \psi_{r+1} e(\bi),&\text{otherwise.}
\end{cases}
\end{align*}
for $\bi,\bj\in I^n$ and all admissible $r$ and $s$. Moreover, $\R$ is naturally
$\Z$-graded with degree function determined by
$$\deg e(\bi)=0,\qquad \deg y_r=2\qquad\text{and}\qquad \deg
  \psi_s e(\bi)=-a_{i_s,i_{s+1}},$$
for $1\le r\le n$, $1\le s<n$ and $\bi\in I^n$.
\end{Defn}

To make the link between the cyclotomic Hecke algebra $\H$ and the quiver Hecke algebra $\R$ suppose
that $R=K$ is a field of characteristic~$p\ge0$. Fix non-zero element $\xi$ of~$K$ and let
$e\in\{0,2,3,4,\dots\}$ is the quantum characteristic of $\xi$. As noted after Definition~\ref{hecke},
if the parameters $\bQ$ are integral then the algebra $\H=\H(e,\charge)$ is determined by $e$ and the
multicharge $\charge$. Let $\Lambda=\Lambda_\charge\in X^+$ be the unique positive weight such that
$$(\Lambda,\alpha_i) = \#\set{1\le l\le\ell | \kappa_l\equiv i\pmod e},
     \qquad\text{ for all }i\in I.$$
We define $\HL=\H(e,\charge)$, where $\Lambda=\Lambda_\charge$.

Next observe that (\ref{Murphy}) together with Theorem~\ref{bases} implies that $\HL$ decomposes into
a direct sum of (simultaneous) generalized eigenspaces for the elements $L_1,\dots,L_n$. Moreover, the
possible eigenvalues for $L_1,\dots,L_n$ are precisely the integers in $I$, if $\xi=1$, and otherwise
they are belong to the set $\set{\xi^i|i\in I}$, if $\xi\ne1$. Hence, the generalized eigenspaces for
these elements are indexed by $I^n$. For each $\bi\in I^n$ let $e(\bi)$ be the
corresponding idempotent in~$\HL$ (or zero if the corresponding eigenspace is zero).

\begin{Theorem}[Brundan-Kleshchev]\label{BK iso}
Suppose that $R=K$ is a field, $\xi\in K$ as above, and that $\Lambda=\Lambda_\charge$. Then
there is an isomorphism of algebras $\R\cong\HL$ which sends
$e(\bi)\mapsto e(\bi)$, for all $\bi\in I^n$ and
\begin{align*}
y_r&\mapsto\begin{cases}
  \Sum_{\bi\in I^n}(1 - \xi^{-i_r}L_r)e(\bi),&\text{if }\xi\ne1,\\[5mm]
  \Sum_{\bi\in I^n}(L_r - i_r)e(\bi),&\text{if }\xi=1.
\end{cases}\\
\psi_s&\mapsto \sum_{\bi\in I^n}(T_r+P_r(\bi))Q_r(\bi)^{-1}e(\bi),
\end{align*}
where $P_r(\bi),Q_r(\bi)\in R[y_r,y_{r+1}]$, for $1\le r\le n$ and $1\le s<n$.
\end{Theorem}

We abuse notation and identify the algebras $\R$ and $\HL$ under this
isomorphism. In particular, we will not distinguish between the homogeneous
generators of $\R$ and their images in $\HL$ under the isomorphism of
Theorem~\ref{BK iso}.

The algebra $\R\cong\HL$ has a unique anti-isomorphism
$\star\map{\R}\R; a\mapsto a^\star$ which fixes each of the homogeneous generators.
We note that the automorphism $\star$ is, in general, not equal to the
anti-automorphism $*$ which fixes each of the (non-homogeneous) generators of~$\HL$ in
Definition~\ref{hecke}.

Until further notice fix a multipartition $\bmu\in\Parts$. If $i\in
I$ then an \textbf{$i$-node} is a triple
$(r,c,l)\in\N^2\times\{1,2,\dots,\ell\}$ such that
$i=c-r+\kappa_l\pmod e$. An $i$-node $A$ is an \textbf{addable
$i$-node} of~$\bmu$ if $A\notin[\bmu]$ and $[\bmu]\cup\{A\}$ is the
diagram of a multipartition. Similarly, an $i$-node $B\in[\bmu]$ is
a \textbf{removable $i$-node} of $\bmu$ if $[\bmu]\setminus\{B\}$ is
the diagram of a multipartition.  Given two nodes $A=(r,c,l)$ and
$B=(s,d,m)$ then $A$ is \textbf{below} $B$, or $B$ is
\textbf{above}~$A$, if either $l>m$, or $l=m$ and $r>s$.

Following Brundan, Kleshchev and Wang~, we make the following definitions.

\def\SetBox#1{\Big\{\vcenter{\hsize34mm\centering#1}\Big\}}

\begin{Defn}[\protect{Brundan, Kleshchev and Wang~\cite[Defn.~3.5]{BKW:GradedSpecht}}]\label{gamma}
  Suppose that $\bmu\in\Parts$ and that $A$ is a removable or addable $i$-node of $\bmu$, for some $i\in I$.
  Define integers
  \begin{align*}
  d_A(\bmu)&=\#\SetBox{addable $i$-nodes of $\bmu$\\ strictly below $A$}
                -\#\SetBox{removable $i$-nodes of $\bmu$\\ strictly below $A$},\\
  \intertext{and}
  d^A(\bmu)&=\#\SetBox{addable $i$-nodes of $\bmu$\\ strictly above $A$}
                -\#\SetBox{removable $i$-nodes of $\bmu$\\ strictly above $A$}.
  \end{align*}
  If $\t$ is a standard $\bmu$-tableau then its \textbf{degree} and \textbf{codegree} are defined
  inductively by setting $\deg\t=0=\codeg\t$, if $n=0$, and if $n>0$ then
  $$ \deg\t=\deg\t_{n-1}+d_A(\bnu)\quad\text{and}\quad
     \codeg\t= \codeg\t_{n-1}+d^A(\bnu)
     $$
  where $A=\t^{-1}(n)$ and $\bnu=\Shape(\t_{n-1})$.
\end{Defn}

If $\t$ is a standard tableau define
$\res(\t)=(\res_{\t}(1),\dots,\res_{\t}(n))\in I^n$, where
$\res_{\t}(k)=c-r+\kappa_l\pmod e$ if $k$ appears in row~$r$ and
column~$c$ of~$\t^{(l)}$, for $1\le k\le n$. This definition is
compatible with our previous definition of $\cont_t(k)$ in the sense
that if $R=K$ then $\cont_\t(k)\equiv\res_\t(k)\pmod e$, if $\xi=1$,
and $\cont_\t(k)=\xi^{\res_\t(k)}$, if $\xi\ne1$.

\begin{Defn}[\protect{\cite[Definitions~4.9, 5.1 and 6.9]{HM:GradedCellular}}]\label{psis}
Suppose that $\bmu\in\Parts$. Let
$\ilmu=\res(\t_\bump)$ and $\imu=\res(\tmu)$ and set
$e_\bmu=e(\imu)$ and $e'_\bmu=e(\ilmu)$ and define
$$ y_\bmu=y_1^{d_1}\dots y_n^{d_n}\quad\text{and}\quad y'_\bmu=y_1^{d'_1}\dots y_n^{d'_n},$$
where $d_m=d_{A_m}(\bmu_m)$, $d'_m=d^{A'_m}(\bmu'_m)$, $\bmu_m=\Shape(\tmu_m)$,
$\bmu'_m=\Shape((\t_{\bmu'})_m)$, and $A_m$ and $A_m'$ are the nodes such that
$\tmu(A_m)=m$ and $\t_{\bmu'}(A_m')=m$.

For the rest of this paper, fix a reduced expression
$d(\u)=s_{i_1}\dots s_{i_k}$ for each row standard $\bmu$-tableau $\u$ and set
$\psi_{d(\u)}=\psi_{i_1}\dots\psi_{i_k}$. Suppose that $(\s,\t)\in\SStd(\bmu)$ and define
$\psi_{\s\t}=\psi_{d(\s)}^\star e_\bmu y_\bmu \psi_{d(\t)}$ and
$\psi'_{\s\t}=\psi_{d(\s)}^\star e'_\bmu y'_\bmu \psi_{d(\t)}$.
\end{Defn}

We warn the reader that, in general, the elements $\psi_{d(\s)}$, $\psi_{d(\t)}$, $\psi_{\s\t}$ and
$\psi'_{\s\t}$ all depend upon the choices of reduced expression for $d(\s)$ and $d(\t)$ that we have
fixed, once and for all, in Definition~\ref{psis}. See \cite[Example~5.6]{HM:GradedCellular} for an
explicit example.

The following result can be viewed as a graded analogue of Theorem~\ref{bases}.

\begin{Theorem}[\protect{Hu-Mathas\cite[Theorems~5.8 and 6.11]{HM:GradedCellular}}]
\label{psi bases}\hspace*{50mm}
\begin{enumerate}
\item $\set{\psi_{\s\t}|(\s,\t)\in\SStd(\Parts)}$ is a graded cellular basis of $\HL$.
\item $\set{\psi'_{\s\t}|(\s,\t)\in\SStd(\Parts)}$ is a graded cellular basis of $\HL$.
\end{enumerate}
In particular, if $(\s,\t)\in\SStd(\Parts)$ then
$$\deg\psi_{\s\t}=\deg\s+\deg\t\quad\text{and}\quad
\deg\psi'_{\s\t}=\codeg\s'+\codeg\t'.$$
\end{Theorem}

Graded cellular algebras were introduced in \cite{HM:GradedCellular}. They are a natural extension of
Graham and Lehrer's~\cite{GL} definition of a cellular algebra to the graded setting.

We close this section by extending results in the last subsection about
strong dominance to the $\psi$ and $\psi'$-bases.

\begin{Lemma}\label{psimu}
Suppose that $R=K$ and $\bmu\in\Parts$. Then
$$\psi_{\tmu\tmu}=\sum_{(\u,\v)\Gedom(\tmu,\tmu)}r_{\u\v}m_{\u\v}
\quad\text{and}\quad
 \psi'_{\tmu\tmu}=\sum_{(\u,\v)\Gedom(\tmu,\tmu)}s_{\u\v}n_{\u\v},
$$
for some $r_{\u\v},s_{\u\v}\in K$ such that $r_{\tmu\tmu}\ne0$ and $\s_{\tmu\tmu}\ne0$.
\end{Lemma}

\begin{proof}
As in
\cite[Definition~4.3]{HM:GradedCellular} we fix a modular system
$(\O,\K,K)$ with parameters $v\in\O$ and
$\bQ^\O=(Q_1^\O,\dots,Q_\ell^\O)\in\O^\ell$ such that $\HK=\HO\otimes_\O\K$ is split semisimple and
$\Hk\cong\HO\otimes_\O K$, where $\HO=\HH^\O_n(v,\bQ^\O)$. By
\cite[Defn.~4.12]{HM:GradedCellular}, there exist elements~$y_\bmu^\O$ and
$e_\bmu^\O$ in $\HO$ such that $e_\bmu y_\bmu=e_\bmu^\O y_\bmu^\O\otimes_\O 1_K$. Moreover,
by \cite[Lemma~4.13]{HM:GradedCellular} there exist scalars $r_\s\in\K$ such that in $\HK$
$$e_\bmu^\O y_\bmu^\O= \sum_{\substack{\s\gedom\tmu\\\res(\s)=\imu}}r_\s f_{\s\s}.$$
Therefore, by Theorem~\ref{tilting}, there exist $r^\K_{\u\v}\in\K$ such that
$$e_\bmu^\O y_\bmu^\O=\sum_{(\u,\v)\Gedom(\tmu,\tmu)}r^\K_{\u\v} m_{\u\v}.$$
However, $e_\bmu^\O y_\bmu^\O\in\HO$ and $m_{\u\v}\in\HO$ for all
$(\u,\v)\in\SStd(\Parts)$ so, in fact, $r^\K_{\u\v}\in\O$. Hence, we can
reduce this equation modulo the maximal ideal of~$\O$ to write
$e_\bmu y_\bmu=\psi_{\tmu\tmu}$ in the required form.

We leave the proof of the formula for $\psi'_{\tmu\tmu}$ to the reader. It
is proved in exactly same way except that \cite[Lemma~6.5]{HM:GradedCellular} is used in place of
\cite[Lemma~4.13]{HM:GradedCellular}.
\end{proof}

\begin{Theorem}\label{psi mn}
Suppose that $R=K$ and $(\s,\t)\in\SStd(\Parts)$.  Then
\begin{xalignat*}{3}
 \psi_{\s\t}&=\Sum_{(\u,\v)\Gedom(\s,\t)}a_{\u\v}\,m_{\u\v}, &
    m_{\s\t}&=\Sum_{(\u,\v)\Gedom(\s,\t)}c_{\u\v}\,\psi_{\u\v},\\
 \psi'_{\s\t}&=\Sum_{(\u,\v)\Gedom(\s,\t)}b_{\u\v}\,n_{\u\v},&
    n_{\s\t}&=\Sum_{(\u,\v)\Gedom(\s,\t)}d_{\u\v}\,\psi'_{\u\v}
\end{xalignat*}
for some scalars $a_{\u\v}$, $b_{\u\v}$, $c_{\u\v}$ and $d_{\u\v}$ in $K$ such that
$a_{\s\t}$, $b_{\s\t}$, $c_{\s\t}$ and $d_{\s\t}$ are all non-zero.
\end{Theorem}

\begin{proof} For convenience, let $\bi^s=\res(\s)$ and $\bi^\t=\res(\t)$.
By Theorem~\ref{BK iso} and
Definition~\ref{psis}, if $d(\t)=s_{i_1}\dots s_{i_k}$ is reduced then
$$\psi_{\tmu\t}=\psi_{\tmu\tmu}(T_{i_1}+P_{i_1}(\bi_1))Q_{i_1}(\bi_1)^{-1}e(\bi_1)\dots
         \dots(T_{i_k}+P_{i_k}(\bi_k))Q_{i_k}(\bi_k)^{-1}e(\bi_k),$$
where $\bi_1=(\imu)^{s_{i_1}}$ and $\bi_{j}=\bi_{j-1}^{s_{i_j}}$ for all $2\leq j\leq k$ with $\Sym_n$
acting on~$I^n$ from the right in the natural way. (Thus, $\bi_k=\bi^\t$.) Using the relations we can
rewrite the expression for $\psi_{d(\t)}e(\bi^{\t})$ as $rT_{d(\t)}e(\bi^{\t})$ plus a linear
combination of terms of the form $LT_u$, where $r\in K$ is non-zero, $u<d(\t)$ and $L\in\<L_1,\dots,L_n\>$.
By Lemma~\ref{psimu} and Corollary~\ref{tilting}, $\psi_{\tmu\tmu}L$ can be written as linear
combinations of elements of form $m_{\u\v}$ with $(\u,\v)\Gedom(\tmu,\tmu)$. Therefore, applying
Lemma~\ref{technical2} to the elements $m_{\u\v}T_u$ shows that $\psi_{\tmu\t}$ can be written as a
linear combination of elements $m_{\u\v}$ with $(\u,\v)\Gedom(\tmu,\t)$.  Using the same argument to
act with $e(\bi^{\s})\psi_{d(\s)}^\star$ from the left shows that $\psi_{\s\t}$ can be written in the
required form.  Arguing by induction on the strong dominance order~$\Gedom$ we can now invert this
equation to show that $m_{\s\t}$ is a linear combination of elements $\psi_{\u\v}$ with
$(\u,\v)\Gedom(\s,\t)$.

The other equations can be proved similarly.
\end{proof}

\begin{Cor}\label{psi strong}
Suppose that $R=K$ and $(\s,\t),(\u,\v)\in\SStd(\Parts)$. Then
$\psi_{\s\t}\psi'_{\u\v}\ne0$ only if $\u'\gedom\t$ and
$\psi'_{\u\v}\psi_{\s\t}\ne0$ only if $\v'\gedom\s$.
\end{Cor}

\begin{proof}
  By Theorem~\ref{psi mn} there exist scalars $a_{\c\d}$ and $b_{\a\b}$ such that
$$\psi'_{\u\v}\psi_{\s\t}
     =\Big(\Sum_{(\c,\d)\Gedom(\u,\v)}b_{\c\d}\,n_{\c\d}\Big)
      \Big(\Sum_{(\a,\b)\Gedom(\s,\t)}a_{\a\b}\,m_{\a\b}\Big).
$$
Therefore, $\psi'_{\u\v}\psi_{\s\t}\ne0$ only if there exist
tableaux $\a,\b,\c$ and~$\d$ such that $n_{\c\d}m_{\a\b}\ne0$ and
$(\c,\d)\Gedom(\u,\v)$ and $(\a,\b)\Gedom(\s,\t)$. By
Corollary~\ref{tilting} this happens only if $\d'\gedom\a$.
Therefore, $\v'\gedom\d'\gedom\a\gedom\s$ as required. The result
for $\psi_{\s\t}\psi'_{\u\v}$ now follows by applying the graded
involution~$\star$.
\end{proof}

We leave the following result as an exercise for the reader. It follows
easily using Theorem~\ref{BK iso} and Corollary~\ref{tilting}.

\begin{Cor}\label{graded tilting}
Suppose that $R=K$ and let $(\s,\t)\in\SStd(\Parts)$. Then there exist
scalars $a_{\u\v},b_{\u\v}\in K$ such that
$$ \psi_{\s\t}y_r=\Sum_{(\u,\v)\Gdom(\s,\t)}a_{\u\v}\,\psi_{\u\v}
\quad\text{and}\quad
\psi'_{\s\t}y_r=\Sum_{(\a,\b)\Gdom(\s,\t)}b_{\a\b}\,\psi'_{\a\b}.
$$
Moreover, $a_{\u\v}$ and $b_{\a\b}$ are non-zero only when
$\res(\u)=\res(\s)$, $\res(\v)=\res(\t)$, $\res(\a)=\res(\s)$, $\res(\b)=\res(\t)$
and $\deg\u+\deg\v=\deg\s+\deg\t+2$, $\codeg\a'+\codeg\b'=\codeg\s'+\codeg\t'+2$.
\end{Cor}

\section{Graded induction and Graded Specht modules}
Before starting the proof of our Main Theorem we recall the facts that we need about $\Z$-graded
algebras and the construction of the graded Specht modules and their duals.

Suppose that $A=\bigoplus_{k\in\Z}A_k$ is a $\Z$-graded algebra. If
$a\in A_k$ then $a$ is homogeneous of \textbf{degree} $\deg a=k$. If
$M=\bigoplus_k M_k$ is a graded $A$-module let $M\<s\>$ be the
graded $R$-module obtained by shifting the grading on $M$ upwards by
$s$; that is, $M\<s\>_k=M_{k-s}$, for $k\in\Z$. Let $\Mod A$ be the
category of finite dimensional graded (right) $A$-module with
homomorphisms being the degree preserving maps (of degree zero).  If
$A$ has a degree preserving anti-involution $\star$  then the
\textbf{contragredient dual} of $M$ is the graded $A$-module
$$M^\circledast = \bigoplus_{d\in\Z}\Hom_A(M\<d\>,K)$$ where the
action of $A$ is given by $(fa)(m)=f(ma^\star)$, for all $f\in
M^\circledast$, $a\in A$ and~$m\in M$.

\subsection{Graded Specht modules}
Following the standard construction from the theory of (graded) cellular
algebras, the two graded cellular bases of Theorem~\ref{psi bases} define
graded cell modules for~$\HL$. More explicitly,
for each multipartition $\bmu\in\Parts$ the \textbf{graded Specht module} $S^\bmu$
and the \textbf{graded dual Specht module} $S_\bmu$ are the graded $\HL$-modules such that
$$S^\bmu\<\deg\tmu\>=(\psi_{\tmu\tmu}+\Hmu)\HL\quad\text{and}\quad
  S_\bmu\<\codeg\t_{\bmu'}\>=(\psi'_{\tmu\tmu}+\Hmu')\HL,$$
where $\Hmu$ is the two-sided ideal of $\HL$ spanned by the elements
$\psi_{\u\v}$, where $(\u,\v)\in\SStd(\bnu)$, $\bnu\gdom\bmu$ and
$\Hmu'$ is spanned by  the $\psi'_{\u\v}$,  for $(\u,\v)\in\SStd(\bnu)$
with $\bnu\gdom\bmu$. Thus, $S^\bmu$ has a natural basis
$\set{\psi_\t|\t\in\Std(\bmu)}$, and $S_\bmu$ has a basis
$\set{\psi'_\t|\t\in\Std(\bmu)}$, where the action on both modules in induced
by the action of $\HL$ upon the $\psi$ and $\psi'$ bases of $\HL$,
respectively, and $\deg\psi_\t=\deg\t$, $\deg\psi'_\t=\codeg\t'$, for $\t\in\Std(\bmu)$.
In particular, by \cite[Prop.~6.19]{HM:GradedCellular}, $S^{\bmu}$ is isomorphic to
the graded Specht module defined by Brundan, Kleshchev and Wang
\cite{BKW:GradedSpecht}. See \cite[\Sect2]{HM:GradedCellular} for more
details.

To explain the relationship between the graded Specht module and its dual we need to recall
the description of the blocks of~$\HL$.  Let $\beta\in Q_+$ be a positive root with
$\sum_{i\in I}(\Lambda_i,\beta)=n$ and
set $I^\beta=\set{\bi\in I^n|\alpha_{i_1}+\dots+\alpha_{i_n}=\beta}$ and
$e_\beta=\sum_{\bi\in I^\beta}e(\bi).$
By \cite[Theorem~2.11]{LM:AKblocks} and \cite[Theorem~1]{Brundan:degenCentre},
$$\HL=\bigoplus_{\beta\in Q_+,\,I^\beta\ne\emptyset}\HL[\beta], \quad\text{where }
           \HL[\beta]=e_\beta\HL,$$
where $\HL[\beta]$ is an indecomposable two-sided ideal of $\HL$ whenever $e_{\beta}\neq 0$. Following
Brundan and Kleshchev~\cite[(3.4)]{BKW:GradedSpecht}, if
$\beta\in Q_+$ with $e_{\beta}\neq 0$ then define the \textbf{defect} of $\HL[\beta]$ by
$$\defect\beta=(\Lambda,\beta)-\frac12(\beta,\beta).$$
Then $\defect\beta\ge0$. Let $\Parts[\beta]=\set{\bmu\in\Parts|\imu\in I^\beta}$ and suppose
that $\bmu\in\Parts[\beta]$. Then, by \cite[Prop.~6.19]{HM:GradedCellular},
$$S^\bmu\cong S_{\bump}^\circledast\<\defect\beta\>,$$
where $S_{\bump}^\circledast$ is the graded dual of $S_{\bump}$. This justifies calling
$S_\bmu$ a graded dual Specht module. For the rest of this section we fix $\beta\in
Q_+$ and $\bmu\in\Parts[\beta]$ such that $e_{\beta}\neq 0$.

To prove our Main Theorem we need to compute $\iInd S^\bmu$, for $i\in I$.
To do this we will use another construction of the graded Specht modules
which, up to shift, realizes them as submodules of $\HL$. Recall that $\tmu$
is the unique standard $\bmu$-tableau such that $\tmu\gedom\t$ for all
$\t\in\Std(\bmu)$. Let $\tlmu$ be the unique standard $\bmu$-tableau such
that $\t\gedom\tlmu$ for all $t\in\Std(\bmu)$. Then $\tlmu$ is the standard
$\bmu$-tableau which has the numbers $1,2,\dots,n$ entered in order down
its columns in the components from right to left. Equivalently, $\tlmu$ is
the tableau conjugate to $\t^{\bump}$. Define $w_\bmu=d(\tlmu)\in\Sym_n$.
It is easy to check that $w_\bmu^{-1}=w_{\bump}=d(\t_{\bump})$.

\begin{Defn}\label{zmu defn}
    Suppose that $\bmu\in\Parts[\beta]$, for $\beta\in Q_+$. Define
    $\zbmu=y'_\bump\psi_{w_\bump}e_\bmu y_\bmu $.
\end{Defn}

Consulting the definitions,
$\zbmu=\psi'_{\tmup\tlmup}\psi_{\tmu\tmu}
      =\psi'_{\tmup\tmup}\psi_{\tlmu\tmu}$. The connection between these
elements and the graded Specht modules is the following.

\begin{Lemma}[\protect{\cite[\Sect6.4]{HM:GradedCellular}}]\label{submodules}
    Suppose that $\bmu\in\Parts[\beta]$. Then, as graded $\HL$-modules,
    $$S^\bmu\<\defect\beta+\codeg\t_{\bmu}\>\cong \zbmu\HL\quad\text{and}\quad
      S_{\bmu'}\<\defect\beta+\deg\tmu\>\cong \zbmu^{\star}\HL.$$
\end{Lemma}

Hence, to determine $\iInd S^\bmu$ it suffices to describe the $\HL[n+1]$-module $\iInd z_\bmu\HL$. To
do this we adapt ideas which Ryom-Hansen~\cite{RyomHansen:GradedTranslation} used to describe the
induced Specht modules of~$\HSym$.

\subsection{Graded induction of Specht modules}
We are now ready to prove our main theorem. We start by describing the
$i$-induction  functors for $\HL$ more explicitly.

Recall that $I=\Z/e\Z$. For each $i\in I$ define
$$e_{i,n}=\sum_{\bj\in I^n}e(\bj\vee i)\in\HL[n+1].$$
Then $\sum_{i\in I}e_{i,n}=\sum_{\bi\in I^{n+1}}e(\bi)$ is the identity element of $\HL[n+1]\cong\R[n+1]$.
Let $\Mod\HL$ be the category of finite dimensional graded $\HL$-modules, with morphisms being
$\HL$-module homomorphisms of degree zero.

\begin{Lemma}\label{embedding}
  Suppose that $R=K$ is field and that $i\in I$. Then
  there is a (non-unital) embedding of graded algebras
  $\HL\hookrightarrow\HL[n+1]$ given by
  $$e(\bj)\mapsto e(\bj\vee i),\quad
    y_r\mapsto e_{i,n}y_r\quad\text{and}\quad
    \psi_s\mapsto e_{i,n}\psi_s,$$
    for $\bj\in I^n$, $1\le r\le n$ and $1\le s<n$. This  map induces an
    exact functor
    $$F_i\map{\Mod\HL}{\Mod\HL[n+1]}; M\mapsto M\otimes_{\HL}e_{i,n}\HL[n+1].$$
  Moreover, $F_i = \iInd_{\HL}^{\HL[n+1]}$ is the graded induction
  functor from $\Mod\HL$ to $\Mod\HL[n+1]$.
\end{Lemma}

\begin{proof}
  The images of the homogeneous generators of $\HL$ under this embedding
  commute with $e_{i,n}$, which implies that this defines a non-unital degree preserving homomorphism
  from $\HL$ to~$\HL[n+1]$. This map is an embedding by Theorem~\ref{BK iso}. The remaining claims follow
  because $e_{i,n}$ is idempotent and $\sum_{i\in I}e_{i,n}$ is the
  identity element of~$\HL[n+1]$.
\end{proof}

We now fix a multipartition $\bmu\in\Parts$ and introduce the notation that
we need to prove our main result. Let $A_1,\dots,A_z$ be the addable
$i$-nodes of $\bmu$ ordered so that $A_{i+1}$ is below~$A_i$, for $1\le
i<z$. Finally, let $\balpha_i$ be the multipartition of~$n+1$ obtained by
adding $A_i$ to $\bmu$, so that $[\balpha_i]=[\bmu]\cup\{A_i\}$, for
$1\le i\le Z$. Then $\balpha_1\gdom\balpha_2\gdom\dots\gdom\balpha_z$ because
we arranged the nodes $A_1,\dots,A_z$ in ``downwards order''. For
notational convenience in what follows we set
$\balpha=\balpha_1$ and $\bomega=\balpha_z$.

The two extremal multipartitions $\balpha$ and $\bomega$ will be particularly
important in what follows: $\balpha$ is the most dominant multipartition obtained
from $\bmu$ by adding an $i$-node and $\bomega$ is the least dominant such
multipartition. Define $\zbmui=e_{i,n}\zbmu$ to be the image of $\zbmu$ under the
algebra embedding of Lemma~\ref{embedding}. Then by Definition~\ref{psis} and
Definition~\ref{zmu defn}
$$\zbmui = e_{i,n} y'_\bump\psi_{w_\bump} e_\bmu y_\bmu e_{i,n}.
$$
Therefore, by Lemma~\ref{submodules}
and Lemma~\ref{embedding}, we have the following.

\begin{Lemma}\label{induced}
  Suppose that $\bmu\in\Parts[\beta]$ and $i\in I$. Then there is an isomorphism of
  graded $\HL[n+1]$-modules,
  $$\iInd S^\bmu\<\defect\beta+\codeg\t_{\bmu}\>\cong \zbmui\HL[n+1].$$
\end{Lemma}

To prove our main theorem we show that $\zbmui\HL[n+1]$ has a filtration by
graded Specht modules. For each $1\le k\le z$, define
$\talpmu$ to be the unique standard $\alpha_k$-tableau such that $(\talpmu)_n=\tlmu$.

\begin{Prop}\label{dominance}
  Suppose that $\bmu\in\Parts[\beta]$ and $i\in I$. Then
  $$\Bmu=\set{e_{i,n}y'_{\bmu'}\psi_{\talpmu\t}|
                    \t\in\Std(\balpha_k)\,\text{ for }1\le k\le z}$$
  is a basis of $\zbmui\HL[n+1]$.
\end{Prop}

\begin{proof} Let $I^\bmu$ be the vector space spanned by $\Bmu$. By \cite[Corollary
5.8]{BK:GradedDecomp} and Lemma \ref{induced}, we know that
$$ \dim\zbmui\HL[n+1]=\dim\iInd S^\bmu=\#\Bmu\ge\dim I^\bmu.  $$
Therefore, to prove the Proposition, it suffices to show that $\zbmui\HL[n+1]\subseteq I^\bmu$.

For each multipartition $\brho$ of $n$ let $\brho^+$ be the
multipartition of $n+1$ obtained by adding its lowest addable node.
Similarly, if $\s\in\Std(\brho)$ then define $\s^+$ to be the unique
standard $\brho^+$-tableau such that $\s^+_n=\s$. Note that it is
not necessarily true that $\res_{\s^+}(n+1)=i$ since the lowest
addable node of $\brho$ need not be an $i$-node.

Suppose that $h\in\HL[n+1]$. We want to show that $\zbmui h\in I^{\bmu}$.  By definition,
$\zbmui=e_{i,n}y'_{\bmu'}\psi_{\tlmu\tmu}e_{i,n}$. Moreover, under the natural embedding
$\HL\hookrightarrow\HL[n+1]$ of ungraded algebras,
$m_{\tlmu\tmu}=m_{\tlmu^+(\tmu)^+}$. Therefore, applying Theorem \ref{psi mn} twice,
$$\zbmui=e_{i,n}y'_{\bmu'}\Big(
  \sum_{\substack{(\s\t)\in\SStd(\Parts[n+1])\\(\s,\t)\Gedom\big(\tlmu^+,(\tmu)^+\big)}}a_{\s\t}\psi_{\s\t}
    \Big)e_{i,n},$$
for some $a_{\s\t}\in K$. Therefore, $\zbmui h$ is equal to a linear combination of elements of the form
$e_{i,n}y'_{\bmu'}\psi_{\s\t}$, where $\s,\t\in\Std(\blam)$, $\res_{\s}(n+1)=i$ and $\blam$ is a
multipartition $n+1$ such that $\blam\gedom\bmu^+$. In particular, $\Shape(\s_n)\gedom\bmu$.

On the other hand, since $e'_{\bmu'}y'_{\bmu'}=\psi'_{\tmup\tmup}$
and $e_{i,n}\in K[y_1,\cdots,y_{n+1}]$, applying Theorem \ref{psi
mn} and Corollary \ref{graded tilting} shows that
$e_{i,n}e'_{\bmu'}y'_{\bmu'}$ is equal to a linear combination of
elements $\psi'_{\u\v}$ such that $\u\gedom(\tmup)^+$ and
$\v\gedom(\tmup)^+$. By Corollary \ref{psi strong},
$\psi'_{\u\v}\psi_{\s\t}\neq 0$ only if $\v'\gedom\s$. Therefore,
$e_{i,n}e'_{\bmu'}y'_{\bmu'}\psi_{\s\t}\neq 0$ only if
$\tlmu\gedom\s_n$ because $\tlmu=(\tmup)'\gedom (\v')_n$.
Therefore, by the last paragraph, $\zbmui
h=e_{i,n}y'_{\bmu'}\psi_{\tlmu\tmu}e_{i,n}h=e_{i,n}e'_{\bmu'}y'_{\bmu'}\psi_{\tlmu\tmu}e_{i,n}h$
is equal to a linear combination of the elements of the form
$e_{i,n}y'_{\bmu'}\psi_{\s\t}$, where $\s,\t\in\Std(\blam)$,
$\Shape(\s_n)\gedom\bmu$, $\res_\s(n+1)=i$, and  $\blam$ is a
multipartition of $n+1$ such that $\tlmu\gedom\s_n$. It follows that
$\s=\t_{\bmu}^{\balpha_j}$ and $\blam=\balpha_j$, for some $j$ with
$1\le j\le z$. Hence, $\zbmui h\in I^\bmu$ and the proof is
complete.
\end{proof}

We can now prove our main theorem.

\begin{proof}[Proof of Main Theorem]
  By Lemma~\ref{induced}, $\iInd S^\bmu\<\defect\beta+\codeg\t_{\bmu}\>\cong\zbmui\HL[n+1]$ so it
  remains to show that $\zbmui\HL[n+1]$ has a suitable Specht filtration. For
  $0\le j\le n$ define
  $$I_j=\<e_{i,n}y'_{\bmu'}\psi_{\talpmu\t}\mid
           \t\in\Std(\balpha_k)\,\text{ for }\,1\le k\le j\>.$$
  Then $I_j$ is a submodule of $\zbmui\HL[n+1]$ by Theorem~\ref{psi bases}  and the proof of Proposition \ref{dominance}.
  Furthermore, the map
  $$S^{\balpha_j}\<2\codeg\t_{\bmu}+\deg\t_{\bmu}^{\balpha_j}\>\cong I_j/I_{j-1};
        \psi_\t\mapsto e_{i,n}y'_{\bmu'}\psi_{\t_{\bmu}^{\balpha_j}\t}+I_{j-1},$$
  for $\t\in\Std(\alpha_j)$ is an isomorphism since $\Bmu$ is linearly
  independent. Notice that this map can also be viewed as left
  multiplication:
  $$\psi_{\t_{\bmu}^{\balpha_j}\t}+\Hmu
           \mapsto e_{i,n}y'_{\bmu'}\psi_{\t_{\bmu}^{\balpha_j}\t}+I_{j-1},$$
  for $\t\in\Std(\balpha_j)$. To complete the proof we need to check that
  the degree shifts in the Specht filtration are as predicted by the Main
  Theorem. We have $\deg\talpmu[j]=\deg\tlmu+d_{A_j}(\bmu)$.
  Moreover, $\codeg\t_{\bmu}+\deg\tlmu=\defect\beta$ by
  \cite[Lemma~3.12]{BKW:GradedSpecht}. Therefore,
  $$
    2\codeg\t_{\bmu}+\deg\talpmu[j]=\defect\beta+\codeg\t_{\bmu}+d_{A_j}(\bmu),
  $$
  so that $S^{\balpha_j}\<\defect\beta+\codeg\t_{\bmu}+d_{A_j}(\bmu)\>\cong I_j/I_{j-1}$. 
  Applying Lemma \ref{induced} now completes the proof of our Main Theorem.
\end{proof}

Finally, we note that we also have the following description of the induced
graded dual Specht module. This can be proved either using the isomorphism
$S^\bmu\cong S_\bump^\circledast\<\defect\beta\>$ from
\cite[Prop.~6.19]{HM:GradedCellular}, or by essentially repeating the
argument above starting with the isomorphism
$S_{\bmu'}\<\defect\beta+\deg\tmu\>\cong z_\bmu^{\star}\HL$ from
Lemma~\ref{submodules}.  We leave the details to the reader.

\begin{Cor}\label{dual Spechts}
  Suppose that $\bmu\in\Parts$ and $i\in I$. Then $\iInd S_{\bmu'}$ has a
  filtration
  $$0=J_{z+1}\subset J_z\subset\dots\subset J_1=\iInd S_{\bmu'},$$
  such that $J_k/J_{k+1}\cong S_{\balpha_k}\<d^{A_k}(\bmu)\>$, for
  $1\le k\le z$, where $\{A_1<A_2<\cdots<A_z\}$ is the set of addable $i$-nodes of 
  $\bmu'$ ordered so that $\balpha_z\gdom\dots\gdom\balpha_1$, where
  $\balpha_k=\bmu'\cup\{A_k\}$ for $1\leq k\leq z$.
\end{Cor}


\end{document}